\documentclass[a4paper,11pt]{amsart}
 \usepackage{amssymb,amsfonts,amsmath}
 \usepackage{setspace}
\usepackage{color}
\usepackage[usenames,dvipsnames,svgnames,table]{xcolor}
\usepackage[colorlinks=true,
            linkcolor=blue,
            urlcolor=blue,
            citecolor=blue]{hyperref}

\def\H{\mathcal{H}}

\def\C{\mathbb{C}}
\def\c2{\mathbb{C}^2}
\def\R{\mathbb{R}}

\def\1{\mathbf{1}}
\def\B{\mathbb{B}}

\def\a{\alpha}

\def\f{\varphi}

\newtheorem{lem}{Lemma}[section]
\newtheorem{prop}[lem]{Proposition}
\newtheorem{defi}[lem]{Definition}
\newtheorem{def/not}[lem]{Definition/Notations}

\newtheorem{thm}[lem]{Theorem}

\newtheorem*{ackn}{Acknowledgements}

\newcommand{\HH}{\mathcal{H}}

\DeclareMathOperator{\Hess}{\mathrm{Hess}} %

\begin{document}
\title[The  metric space of  plurisubharmonic functions]{Geometry and Topology of the space of plurisubharmonic functions}

\author{Soufian ABJA}

 \date{\today}

 \address{Ibn Tofail University\\
Kenitra Morocco}
\email{Soufian.abja@uit.ac.ma}

\begin{abstract}
Let $\Omega$ be a strongly pseudoconvex domain. We introduce the Mabuchi space of strongly plurisubharmonic functions in 
$\Omega$. We study metric  properties of this space   using Mabuchi geodesics and establish
  regularity properties of the latter, especially in the ball.
As an application we study the existence of local K\"ahler-Einstein metrics.
\end{abstract}

\maketitle

\tableofcontents


\section*{Introduction}
\label{sec:intro}

 Let $Y$ be a compact K\"ahler manifold and $\a_Y \in H^{1,1}(Y,\R)$ a K\"ahler class. The space $\HH_{\a_Y}$ of K\"ahler metrics $\omega_Y$ in $\a_Y$ can be seen as an infinite dimensional riemannian manifold
whose tangent spaces $T_{\omega_Y} \H_{\a_Y}$ can all be identified with ${\mathcal C}^{\infty}(Y,\R)$.
Mabuchi has introduced in \cite{Mab87} an $L^2$-metric on $\H_{\a_Y}$, by setting
$$
\langle f, g \rangle_{\omega_Y}:=\int_Y f \, g \, \frac{{\omega_Y}^n}{V_{\a_Y}},
$$
where $n=\dim_\C Y$ and $V_{\a_Y}=\int_Y {\omega_Y}^n=\a_Y^n$ denotes the volume of $\a_Y$.
Mabuchi  studied the corresponding geometry of $\H_{\a_Y}$, showing in particular that it
can formally be seen as a locally symmetric space of non positive curvature. The (geometry) metric study of the space $(\H_{\a_Y},\langle , \rangle_{\omega_Y})$ has motiveted a lot of interesting works in the last decades, see notably \cite{Don99,Chen00,CC02,CT08,Chen09,LV11,DL12,Dar13,Dar14,Dar15}.

  The purpose  of this article is to extend some of  these studies to the case when
  $Y$ is a smooth strongly pseudoconvex bounded domain of $\C^n$. We note here that this problem of extension to the local case recently been considered  by Rashkovskii \cite{Rash16} and Hosono \cite{Hos16}, Rashkovskii  studied  geodesics for plurisubharmonic functions in the Cegrell class $\mathcal{F}_1$ on a
bounded hyperconvex domain, he also showed that functions with strong singularities generally cannot be connected by (sub)geodesic arcs. Hosono described the behavior of the  weak geodesics between toric psh functions with poles at the origin.

 Our  first interest is the geometry of the space of plurisubharmonic functions, We equipped the space of plurisubharmonic functions with a Levi-Civita connection $D$ and we describe the tensor curvature and sectional curvature as in a paper of Mabuchi \cite{Mab87}.
Our first main result is to the establish that the space of plurisubharmonic functions is a locally symmetric space:

\medskip
\noindent {\bf Theorem A.}
\textit{The  Mabuchi space  $\mathcal{H}$  equipped with the Levi-Civita connection  $D$ is a locally symmetric space.}
\medskip

 Following  the work of Donaldson \cite{Don99} and Semmes \cite{Sem92} in the compact setting, we reinterpret the geodesics as a solution to a homogeneous complex Monge-Amp\`ere equation. Weak geodesics are introduced as an envelope of functions:
$$\Phi(z,\zeta)=\sup\{u(z,\zeta)/ u \in \mathcal{F}(\Omega\times A,\Psi)\}$$
Our second  main result is to establish regularity properties of geodesics in the ball by adapting the celebrated result of Bedford-Taylor\cite{BT76}:

\medskip
\noindent{ \bf Theorem B.}
\textit{
Let $\mathbb{B}$ be the  unit ball in $\mathbb{C}^n$. Let $\varphi_0$ and  $\varphi_1$ be the end geodesic  points which are $C^{1,1}$. Then the Perron-Bremermann envelope
$$\Phi(z,\zeta)=\sup\{u(z,\zeta)/ u \in \mathcal{F}(\Omega\times A,\Psi)\}$$
admits second-order partial derivates almost everywhere with respect to variable $z\in\B$ which  locally bounded uniformly with respect to $\zeta \in A$  , i.e  for any compact subset  $K\subset \mathbb{B}$ there exists  $C$ which depend on $K,\varphi_{0}$ and  $\varphi_{1}$ such that   $$\Vert D^2_z\Phi\Vert_{ L^{\infty}(K\times A)}\leq C.$$}
\medskip

The existence of local K\"ahler-Einstein metrics was studied by  Guedj, Kolev and Yeganefar \cite{GKY13} in  bounded smooth strongly pseudoconvex domains which are circled. This is equivalent to the resolution of the following Dirichlet problem 
$$(MA)_1\quad\left\{
\begin{array}{ll}
  (dd^c \varphi)^n=\frac{e^{- \varphi} \mu}{\int_{\Omega} e^{-\varphi} d\mu }, & \hbox{in $\Omega$} \\
      \varphi=0 ,& \hbox{ on $\Omega$}
      \end{array}
    \right.
   $$       
They treated also the following family of  Dirichlet problems 
$$(MA)_t\quad\left\{
\begin{array}{ll}
  (dd^c \f_t)^n=\frac{e^{-t \f_t} \mu}{\int_{\Omega} e^{-t \f_t} d\mu } ,& \hbox{in $\Omega$} \\
      \f_t=0, & \hbox{ on $\Omega$}
      \end{array}
    \right.
   $$  
showing that there is a solution for $t< (2n)^{1+1/n}(1 + 1/n)^{(1+1/n)}$. We apply our  study of the geodesics problem and an idea of \cite{DR15,DG16} to  prove that the  existence of a solution to $(MA)_t$ is equivalent to the  coercivity  of the Ding functional:
 
 \medskip
\noindent{\bf Theorem C.} 
\textit {Let $\Omega\subset \mathbb{C}^n$ be a smooth strongly pseudo-convex  circled domain. If there exists $\varepsilon(t), M(t)>0$ such that,
$$\mathcal{F}_{t}(\psi)\leq \varepsilon(t) E(\psi)+M(t)\;\; \;\;\forall \psi \in \mathcal{H},$$
then $(MA)_t$ admits a $S^1$-invariant smooth strictly plurisubharmonic function solution.\\
Conversely if $(MA)_t$ admits such a solution $\varphi_t$ and $\Omega$ is strictly $\varphi_t$-convex, then there exists $\varepsilon(t), M(t)>0$ such that,
$$\mathcal{F}_{t}(\psi)\leq \varepsilon(t) E(\psi)+M(t)\;\; \;\;\forall \psi \in \mathcal{H}.$$ }
\medskip

The organization of the paper is as follows.
\begin{itemize}
  \item Section 1 is devoted to preliminary results and the definition of the space $\mathcal{H}$ and its geometry. 
  \item In Section 2 we show that the  geodesics are continuous (sometimes even Lipschitz)
   up to the boundary of $\Omega\times A$.
   \item In section 3 we prove theorem B.
  \item Finally, we prove Theorem C in Section 4.
  \end{itemize}
\begin{ackn}It is a pleasure to thank my supervisors Vincent Guedj and Said Asserda, for their support, suggestions and encouragement. I thank  Ahmed Zeriahi for very useful discussions and 
suggestions. Also, i would like to thank  Tat Dat T\^{o}  and Zakarias Sj\"{o}str\"{o}m Dyrefelt for a very careful reading of the preliminary version of this paper and very useful discussions.
\end{ackn}
\section{Mabuchi geometry in pseudoconvex domains}
In this section we will study the geometry of the space of plurisubharmonic functions in strongly pseudoconvex domain, based upon works of Mabuchi \cite{Mab87}, Semmes\cite{Sem92} and Donaldson \cite{Don99}, as it was clarified through lecture notes of Guedj \cite{G14} and  Kolev \cite {Kol12}.
\subsection{Preliminaries}
In this section we recall some analytic tools which will be used in the sequel.
Let $\Omega\Subset\mathbb{C}^n$ be a smooth  pseudoconvex bounded domain.
Recall that a bounded domain $\Omega \Subset \mathbb{C}^ n$ is strictly pseudoconvex if  there exists a smooth
function  $\rho$ defined  in neighbourhood  $\Omega '$ of $\bar{\Omega}$ 
such that  $\Omega=\{z\in \Omega' \;\;/ \rho(z)<0\}$ with $dd^c \rho >0$,
where
$$
d:= \partial+ \bar{\partial}\;,\; d^c:={i\over 2\pi}(\partial-\bar{\partial})
$$
\begin{defi}
We let $PSH(\Omega)$ denote the set of plurisubharmonic functions in $\Omega$.
In particular a function $\varphi \in PSH(\Omega)$ is $L^1_{loc}$, upper semi-continuous and such that
$$
dd^c\varphi\geq 0
$$
in the weak sense of positive currents.
\end{defi}

The following cone of "test functions" has been introduced by 
Cegrell \cite{Ceg98}:
 \begin{defi}\cite{Ceg98}
 We let $\mathcal{E}_0(\Omega)$ denote the convex cone of all bounded plurisubharmonic functions $\varphi$ defined in $\Omega$ such that $\lim_{z\to \zeta}\varphi(\zeta)=0$, for every $\xi\in\partial\Omega$, and $\int_{\Omega}(dd^c\varphi)^n<+\infty$.
 \end{defi}
 
 \begin{defi}\cite{Ceg98}
 The class  $\mathcal{E}^p(\Omega)$ is a set of functions $u$ for
which there exists a sequence of functions $u_j \in\mathcal{E}_0(\Omega)$decreasing
towards $u$ in all of $\Omega$, and so that $\sup_j\int_{\Omega}(-u_j)^p (dd^cu_j)^n<+\infty$.
 \end{defi}
 We will need the following 
 maximum principle:
 \begin{prop}\cite{BT76}
 Let $u$,$v$ be  locally bounded plurisubharmonic functions in $\Omega$ such that $\liminf_{z\to \partial \Omega}(u-v)\geq 0$ .
 Then
$$
(dd^c u)^n\leq  (dd^c v)^n \Longrightarrow  u\leq v \;in\; \Omega 
.$$
 \end{prop}
 \subsection{The space of plurisubharmonic potentials}
We begin this section by defining the Mabuchi space of plurisuharmonic functions in $ \Omega$. 
\begin{defi} 
The  Mabuchi space  of plurisubharmonic functions in $\Omega$ is:
$$
\mathcal{H}:=\{\varphi\in C^{\infty} (\bar{\Omega},\mathbb{R})/ dd^c\varphi>0 \; in\; \bar{\Omega}\;  \;\varphi=0\; on\;\partial\Omega\}
$$
\end{defi}
We now  consider the tangent space of  $\mathcal{H}$ in every  
$C^{\infty}(\bar{\Omega},\mathbb{R})$.
\begin{defi}
The tangent space of  $\mathcal{H}$ at point $\varphi$, we denote by $T_{\varphi}\mathcal{H}$  is the linearisation of 
$\mathcal{H}$  defined by:
$$
T_{\varphi}\mathcal{H}=\{\gamma'(0)/ \;\;\varphi :[-\varepsilon,\varepsilon]\rightarrow \mathcal{H}\;\; and\;\; \gamma(0)=\varphi\}
.$$
\end{defi}
 The tangent space of $\mathcal{H}$ at $\varphi$ can be identified with 
 $$
 T_{\varphi}\mathcal{H}\cong\{\xi \in C^{\infty}(\bar{\Omega} ,\mathbb{R})\; /\; \;\xi=0 \;\; on  \; \partial \Omega\}
 $$

Indeed. Let $\xi \in \{\xi \in C^{\infty}(\bar{\Omega} ,\mathbb{R})\; /\; \;\xi=0 \; on  \; \partial \Omega\}$,  
we put $\gamma(s):=\varphi+s\xi $ for  $s$ close enough to  $0$ we have   $\gamma_s\in \mathcal{H}$, and
$$
\gamma(0)=\varphi\;\; and \;\;\gamma'(0)=\xi
$$ this implies that $\xi\in T_{\varphi}\mathcal{H}$ hence
$$
\{\xi \in C^{\infty}(\bar{\Omega},\mathbb{R})\; /\; \;\xi|_{\partial \Omega}=0\}\subset T_{\varphi}\mathcal{H}
.$$
Conversely, let  $\gamma\in \mathcal{H}$ which  gives  $\gamma_t|_{\partial\Omega}=0 $ 
 for every  $t$. In particular
 $\dot{\gamma}(0)|_{\partial\Omega}=0$,
therefore $$\xi=\dot{\gamma}(0)\in \{\xi \in C^{\infty}(\bar{\Omega} ,\mathbb{R})\; /\; \;\xi=0\; on  \; \partial \Omega\}.$$
\begin{defi}\cite{Mab87} 
The Mabuchi metric is the $L^2$ Riemanniann  metric. It is defined by
$$
<<\psi_1,\psi_2>>_{\varphi}:=\int_{\Omega}\psi_1\psi_2 (dd^c\varphi)^n ,$$
where $\varphi\in \mathcal{H},\psi_1,\psi_2\in T_{\varphi}\mathcal{H}$.
\end{defi}
   \subsection{Mabuchi geodesics}
   Geodesics between two points $\varphi_0$, $\varphi_1$ in $\mathcal{H}$ are defined as the extremals of
the Energy functional
$$
\varphi\longmapsto H(\varphi):= \frac{1}{2}\int_{0}^1\int_{\Omega}(\dot{\varphi}_t)^2( dd^c\varphi_t)^n
.$$
where $\varphi=\varphi_t$ is a path in $\mathcal{H}$ joining $\varphi_0$ to $\varphi_1$. The geodesic equation is obtained by computing the Euler-Lagrange equation of the functional $H$.
\begin{thm} 
The geodesic equation is
\begin{equation}\label{1}
    \ddot{\varphi}(t)-|\nabla\;\dot{\varphi}(t)|^{2}_{\varphi(t)}=0\;\;
\end{equation}
where $\nabla$ is the gradient  relative to the metric $\omega_{\varphi}=dd^c\varphi$.
\end{thm}
\begin{proof}
We need to compute the Euler-Lagrange equation of the Energy functional. Let $(\Phi_{s,t})$ be a variation of $\varphi$
with fixed end points,
$$\phi_{0,t}=\varphi_t, \phi_{s,0}=\varphi_0, \phi_{s,1}=\varphi_1\; and \;\phi_{s,t}=0 \; on \; \partial \Omega$$
Set $\psi_t:=\frac{\partial \phi}{\partial s}|_{s=0}$ and observe that $\psi_0\equiv\psi_1\equiv 0$ and $\psi_t=0$ on $\partial \Omega$. Thus
$$\phi_{s,t}=\varphi_t+s\psi_t+\circ(s)\; and \; \frac{\partial \phi_{s,t}}{\partial t}=\dot{\varphi}+s\dot{\psi_t}+\circ(s)$$
and 
 $$(dd^c\phi_{s,t})^{n}=(dd^c(\varphi_t+s\psi_t))^n=(dd^c\varphi_t)^n +s.n dd^c\psi_t\wedge( dd^c\varphi_t)^{n-1}.$$
 A direct computation yields 
 \begin{eqnarray*}
 H(\phi_{s,t})&=&\frac{1}{2}\int_0^1\int_{\Omega}(\dot{\phi_{s,t}})^2 (dd^c\phi_{s,t})^n dt\\
 &=&  H(\varphi_{t})+s\int_0^1\int_{\Omega}\dot{\varphi_t}\dot{\psi}(dd^c\varphi_t)^n dt \\
 &&+\frac{ns}{2}\int_0^1\int_{\Omega}\dot{\varphi_t}^2dd^c\psi_t\wedge( dd^c\varphi_t)^{n-1}dt.
 \end{eqnarray*}
 Integration by part, and the fact $\psi_0\equiv\psi_1\equiv 0$  yields
 $$\int_0^1\int_{\Omega}\dot{\varphi_t}\dot{\psi}(dd^c\varphi_t)^n dt=-\int_0^1\int_{\Omega}\psi_t\{\ddot{\varphi_t}(dd^c\varphi_t)^n +n \dot{\varphi_t}dd^c\dot{\varphi_t}\wedge (dd^c\varphi_t)^{n-1}\} dt.$$ 
 And we have also by Stokes and the fact $\dot{\varphi_t}=0$ on $\partial \Omega$
 $$\int_0^1 \int_{\Omega}(\dot{\varphi})^2 dd^c\psi_t\wedge (dd^c\varphi_t)^{n-1}dt=2\int_0^1 \int_{\Omega}\psi_t(d\dot{\varphi}\wedge d^c\dot{\varphi_t}+\dot{\varphi_t}dd^c\dot{\varphi_t}\wedge (dd^c\varphi_t)^{n-1})dt$$
 hence $$
 H(\varphi_{s,t})= H(\varphi_{t})+s\int_0^1\int_{\Omega}\psi_t\left\{ -\ddot{\varphi_t}(dd^c\varphi_t)^n+ n d\dot{\varphi_t}\wedge d^c\dot{\varphi}\wedge(dd^c\varphi_t)^{n-1} \right\}dt +\circ(s)$$
 which implies 
 \begin{eqnarray*}
 &0=& d_{\varphi_t}H.\psi_t\\
 &=&\lim_{s \to 0}\frac{H(\varphi_{s,t})-H(\varphi_t)}{s}\\
 &=&\int_0^1\int_{\Omega}\psi_t\left\{ -\ddot{\varphi_t}(dd^c\varphi_t)^n+ n d\dot{\varphi_t}\wedge d^c\dot{\varphi}\wedge(dd^c\varphi_t)^{n-1} \right\}dt.
 \end{eqnarray*}
 Therefore $(\varphi_t)$ is critical point of $H$ if and only if
$$\ddot{\varphi_t}(dd^c\varphi_t)^n=n d\dot{\varphi_t}\wedge d^c\dot{\varphi}\wedge(dd^c\varphi_t)^{n-1}.$$
\end{proof}
\subsection{Levi-Civita connection }
As for Riemanniann manifolds of finite dimension. One can find the local expression of the Levi-Civita connection by polarizing the geodesic equation.
\begin{defi} 
We define the covariant derivative of the vector field $\psi_t$ along the path  $\varphi_t$ in $\mathcal{H}$ by the formula
$$D\psi:=\frac{d\psi}{dt}-< \nabla\psi,\nabla\;\dot{\varphi}>_{\varphi}$$
\end{defi}
\begin{thm} 
$D$ is the Levi-Civita connection.
\end{thm}
\begin{proof} 
To show that $D$ is a Levi-Civita connection, we must show that the connection $D$ is metric-compatible and a  torsion-free.\\
i) Metric-compatibility: Let $\psi_1 $, $\psi_2$  be two vector fields 
\begin{eqnarray*}
  {d \over dt}<<\psi_1,\psi_2>>_{\varphi} &=& {d\over dt}\int _{\Omega}\psi_1\psi_2 (dd^c\varphi)^n \\
   &=& \int _{\Omega}(\dot{\psi_1}\psi_2+\psi_1 \dot{\psi_2})(dd^c\varphi)^n+n\psi_1\psi_2 dd^c\dot{\varphi}\wedge (dd^c\varphi)^{n-1} \\
   &=&\int_{\Omega}(\dot{\psi_1}\psi_2+\psi_1\dot{\psi_2}-<\nabla\; (\psi_1\psi_2),\nabla\; \dot{\varphi}>_{\varphi})(dd^c\varphi)^n  \\
   &=&\int_{\Omega}((\dot{\psi_1}-<\nabla\psi_1,\nabla\dot{\varphi}>_{\varphi})\psi_2(dd^c\varphi)^n\\
   &&+\int_{\Omega}\psi_1(\dot{\psi_2}-<\nabla\; \psi_2,\nabla\;\dot{\varphi}>_{\varphi})(dd^c\varphi)^n \\
   &=& <<D \psi_1,\psi_2>>_{\varphi}+<<\psi_1,D \psi_2>>_{\varphi}
   \end{eqnarray*}
   
(The passage from the second line to the third line is a result  of the equation
$$d( \psi_1\psi_2 d^c\dot{\varphi}\wedge (dd^c\varphi)^{n-1})=d(\psi_1\psi_2)\wedge d^c\dot{\varphi}\wedge (dd^c\varphi)^{n-1}+\psi_1\psi_2dd^c\dot{\varphi}\wedge (dd^c\varphi)^{n-1}$$ and Stokes theorem).\\
ii) $D$ is a torsion-free, because
$$D_s\frac{d \varphi}{ dt}=D_t\frac{d\varphi}{ ds}.$$

Thus $D$ is a Levi-Civita connection.
\end{proof}
\subsection{ Curvature tensor}
We will define  the  curvature tensor and the  sectional curvature  and we will give those expressions. 
We will finish by proving that the space of plurisubharmonic functions is locally symmetric. 
We start by giving some definitions and conventions.
\begin{defi}
Let  $\psi$ and $\theta$ be two functions in the tangent space of $\mathcal{H}$ at $\varphi$.
The Poisson  bracket of $\psi$ and $\theta$ compared to the  form $\omega_{\varphi}=dd^c\varphi$ is
$$
\{\psi,\theta\}=\{\psi,\theta\}_{\varphi}:=i\sum_{\alpha,\beta=1}\varphi^{\alpha\bar{\beta}}\left(\frac{\partial\psi}{\partial \bar{z}_{\beta}}\frac{\partial \theta}{\partial z_{\alpha}}-\frac{\partial \psi}{\partial z_{\alpha}}\frac{\partial\theta}{\partial \bar{z}_{\beta}}\right)
$$
where $(\varphi^{\alpha\bar{\beta}})$ is the inverse matrix of  $(\varphi_{\alpha\bar{\beta}})$.
\end{defi}
\begin{lem}\label{poi}
Let $\psi$, $\theta$ and $\eta$  three functions  belonging to the tangent space of  $\mathcal{H}$ at $\varphi$. The Poisson bracket satisfies the following properties :

i)\;$\{\psi,\theta\}=-\{\theta,\psi\}$

ii) \;$\{\psi,\theta\}=\omega_{\varphi}(X_{\psi},X_{\theta})$

iii)\; $\{\psi,\theta+\eta\}=\{\psi,\theta\}+\{\psi,\eta\}$

iv)\;$[X_{\psi},X_{\theta}]=X_{\psi}(X_{\theta})-X_{\theta}(X_{\psi})=X_{\{\psi,\theta\}}$

v)\;$\int_{\Omega}\{\psi,\theta\}\eta (dd^c\varphi)^n=\int_{\Omega}\psi\{\theta,\eta\}(dd^c\varphi)^n$

vi)\;$D\{\psi,\theta\}=\{D\psi,\theta\}+\{\psi,D\theta\}$. \\
where  $X_{\psi}:=i\nabla\psi$ and $[,]$ is the Lie bracket.\\
\end{lem}
Let $\psi$ be a function in tangent space, the Hessian of $\psi$ is defined by 
$\Hess \psi=\nabla^{\varphi} d\psi$, where $\nabla^{\varphi}$ is the Levi-Civita connection respectively to the form $\omega_{\varphi}=dd^c\varphi$. We Recall  in the next lemma some proprieties of the Hessian well know in the literature.
\begin{lem}\label{hess}Let $X$ and $Y$ be two vector fields. Then the Hessian satisfies the following proprieties:  \\
i)\;\; $\Hess\psi(X,Y)=<\nabla^{\varphi}_X\nabla^{\varphi}\psi,Y>_{\varphi}$.\\
ii)\;\;$\Hess\psi(X,Y)=X(Y(\psi))-\nabla^{\varphi}_XY(\psi)$.\\
iii)\;\;$ dd^c\psi(X,iY)=\Hess\psi(X,Y)+\Hess\psi(iX,iY).$\\
Where $\nabla^{\varphi}$ and $<,>_{\varphi}$ are the Levi-Civita connection and the metric respectively associated to the form $\omega_{\varphi}=dd^c\varphi$.
\end{lem}
In the sequel of this section, we consider a 2-parameters family $\varphi(t,s)\in \mathcal{H}$ and a vector field 
$\psi(t,s)\in T_{\varphi}\mathcal{H}$ defined along $\varphi$. We denote by
$$\varphi_t=\frac{d\varphi}{dt},\quad \varphi_s=\frac{d\varphi}{ds}
.$$
\begin{defi} 
The curvature tensor of the Mabuchi metric  in $\mathcal{H}$ is defined by
\begin{equation*}
    R_{\varphi}(\varphi_t,\varphi_s)\psi:=D_{t}D_{s}\psi-
D_{s}D_{t}\psi
\end{equation*}
where $\varphi(s,t)\in \mathcal{H}$ is 2-parameters family and vector field $\psi(s,t)\in T_{\varphi}\mathcal {H}$.\\
The  sectional curvature is given by
  \begin{equation*}
K_{\varphi}(\varphi_t,\varphi_s):=<<R_{\varphi}(\varphi_t,\varphi_s)\varphi_t,\varphi_s>>_{\varphi}
\end{equation*}
\end{defi}
\begin{thm} The curvature tensor of the Mabuchi metric  in $\mathcal{H}$  can be expressed as
 $$
 R_{\varphi}(\varphi_t,\varphi_s)\psi=-\{\{\varphi_t,\varphi_s\},\psi\}
 .$$
 The sectional curvature is the following
  $$
  K_{\varphi}(\varphi_t,\varphi_s)=-||\{\varphi_t,\varphi_s\}||^2_{\varphi}\leq 0
  ,$$
where  $\{,\}_{\varphi}$ is the Poisson bracket associate  to the form $\omega_{\varphi}=dd^c\varphi$.
 \end{thm}
\begin{proof} 
To compute the curvature  tensor of $D$, we compute the first term in the definition of the   curvature tensor . Indeed, let $\psi$ be the vector field, its derivative along the path $\varphi_s$
$$
D_s \psi=\psi_s-<\nabla \psi,\nabla\varphi_s>_{\varphi}=\psi_s+\Gamma_{\varphi}(\psi,\varphi_s)
,$$
where  $$\Gamma_{\varphi}(\psi,\varphi_s)=-<\nabla \psi, \nabla \varphi_s>_{\varphi},$$
we derive the $D_s \psi$ along the path $\varphi_t$ as follows
\begin{eqnarray*}
  D_{t}D_{s}\psi &=& D_{t}(\psi_s+\Gamma_{\varphi}(\psi,\varphi_s)\\
   &=&\frac{d}{dt}(\psi_s+\Gamma_{\varphi}(\psi,\varphi_s))+\Gamma_{\varphi}(\psi_s+\Gamma_{\varphi}(\psi,\varphi_s),\varphi_t))\\
   &=& \psi_{st}+\frac{d}{dt}(\Gamma_{\varphi}(\psi,\varphi_s))+\Gamma_{\varphi}(\psi_s,\varphi_t)+\Gamma_{\varphi}(\Gamma_{\varphi}(\psi,\varphi_s),\varphi_t).
\end{eqnarray*}
We  express the second term in RHS  of the last equation:
\begin{eqnarray*}
 \frac{d}{dt}\Gamma_{\varphi}(\psi,\varphi_s) &=& \frac{d}{dt}(-<\nabla\psi,\nabla\varphi_s>_{\varphi}) \\
   &=& -\frac{d}{dt}\varphi^{\alpha\bar{\beta}}\varphi_{s\alpha}\psi_{\bar{\beta}}\\
   &=& -\varphi^{\alpha\bar{\beta}}\varphi_{s\alpha t}\psi_{\bar{\beta}}-\varphi^{\alpha \bar{\beta}}\varphi_{s\alpha}\psi_{\bar{\beta}t}+\varphi^{\alpha\bar{m}}
   \varphi^{n\bar{\beta}}\varphi_{n\bar{m}t}\varphi_{s\alpha}\psi_{\bar{\beta}}\\
   &=&\Gamma_{\varphi}(\psi,\varphi_{ts})+\Gamma_{\varphi}(\psi_t,\varphi_s)+dd^c\varphi_t(\nabla\;\varphi_s,i\nabla\;\psi).
\end{eqnarray*}
By applying of the three properties of lemma \ref{hess} by taken $X=\nabla\varphi_s$ and $Y= \nabla\psi$, then we express the last term in the last equation as follows: $$  
 dd^c\varphi_t(\nabla\, \varphi_s,i\nabla\;\psi)=\hbox{Hess}(\varphi_t)(\nabla\;\varphi_s, \nabla\psi)+\hbox{Hess}(\varphi_t)(i\nabla\;\varphi_s,i\nabla\;\psi).$$
Which gives 
{\small
\begin{equation*}
{d \over dt}\Gamma _{\varphi}(\varphi_t,\psi)=\Gamma_{\varphi}(\psi,\varphi_{ts})+\Gamma_{\varphi}(\psi_s,\varphi_t)+
\hbox{Hess}(\varphi_t)(\nabla\;\varphi_s, \nabla\psi)+\hbox{Hess}(\varphi_t)(i\nabla\;\varphi_s,i\nabla\;\psi).
\end{equation*}
}
We develop the fourth term in the RHS in the last equation  by applying the second proprieties of the lemma \ref{hess},  taken  $X=\nabla \varphi_s$ and $Y=\nabla\psi$:
\begin{eqnarray*}
  \hbox{Hess}(\varphi_t)(\nabla\;\varphi_s,\nabla\;\psi) &=& \nabla\;\varphi_s(\nabla \;\psi(\varphi_t))-(\nabla^{\varphi}_{\nabla\;\varphi_s}\nabla\;\psi)(\varphi_t)\\
   &=& \nabla\;\varphi_s(<\nabla\;\varphi_t,\nabla\;\psi>_{\varphi})- <\nabla\;\varphi_t,\nabla^{\varphi}_{\nabla\;\varphi_s}\nabla\;\psi>_{\varphi}\\
   &=& \Gamma_{\varphi}(\Gamma_{\varphi}(\varphi_t,\psi),\varphi_s)-\hbox{Hess}(\psi)(\nabla\;\varphi_s,\nabla\;\varphi_t)
\end{eqnarray*}
We have  also by applying the first proprieties of lemma \ref{hess}:
\begin{eqnarray*}
  \hbox{Hess}(\varphi_t)(i\nabla\;\varphi_s,i\nabla\;\psi)&=&<\nabla^{\varphi}_{i\nabla\;\varphi_s}\nabla\;\varphi_t,i\nabla\;\psi>_{\varphi}\\
  &=& <\nabla^{\varphi}_{i\nabla\;\varphi_s}(i\nabla\;\varphi_t),i(i\nabla\;\psi)>_{\varphi}\\
  &=&\omega_{\varphi}(\nabla^{\varphi}_{X_{\varphi_s}}X_{\varphi_t},X_{\psi})
\end{eqnarray*}
Where $X_{h}=i\nabla\; h$. Then we have:

\begin{eqnarray*}
\frac{d}{dt}\Gamma_{\varphi}(\varphi_s,\psi)&=&\Gamma_{\varphi}(\psi,\varphi_{ts})+\Gamma_{\varphi}(\psi_t,\varphi_s)+\Gamma_{\varphi}
(\Gamma_{\varphi}(\varphi_s,\psi),\varphi_t)\\
&&
-\Hess(\psi)(\nabla\;\varphi_t,\nabla\;\varphi_s)+
\omega_{\varphi}(\nabla^{\varphi}_{X_{\varphi_t}}X_{\varphi_s},X_{\psi})
\end{eqnarray*}

 After all previous equations, we get the expression of  $D_{t}D_{s}\psi$ as follows:
 {\footnotesize
\begin{eqnarray*}
D_{t}D_{s}\psi &=&\psi_{st}+\Gamma_{\varphi}(\psi,\varphi_{ts})+\Gamma_{\varphi}(\psi_t,\varphi_s)+\Gamma_{\varphi}
(\Gamma_{\varphi}(\varphi_s,\psi),\varphi_t)
-\hbox{Hess}(\psi)(\nabla\;\varphi_t,\nabla\;\varphi_s)\\
&&+\;\omega_{\varphi}(\nabla^{\varphi}_{X_{\varphi_t}}X_{\varphi_s},X_{\psi})+\Gamma_{\varphi}(\psi_s,\varphi_t)+\Gamma_{\varphi}(\Gamma_{\varphi}(\psi,\varphi_s),\varphi_t)
\end{eqnarray*}
}
We get the expression of  $D_{s}D_{t}\psi$ by reversing the roles of   $t$ and $s$ as follows:
{\footnotesize
\begin{eqnarray*}
D_{s}D_{t}\psi &=&\psi_{st}+\Gamma_{\varphi}(\psi,\varphi_{st})+\Gamma_{\varphi}(\psi_s,\varphi_t)+\Gamma_{\varphi}
(\Gamma_{\varphi}(\varphi_t,\psi),\varphi_s)
-\hbox{Hess}(\psi)(\nabla\;\varphi_s,\nabla\;\varphi_t)\\
&&+\;\omega_{\varphi}(\nabla^{\varphi}_{X_{\varphi_s}}X_{\varphi_t},X_{\psi})+\Gamma_{\varphi}(\psi_t,\varphi_s)+\Gamma_{\varphi}(\Gamma_{\varphi}(\psi,\varphi_t),\varphi_s)
\end{eqnarray*}
}
Therefore we get
\begin{eqnarray*}
R_{\varphi}(\varphi_t,\varphi_s)\psi&=&D_{t}D_{s}\psi-
D_{s}D_{t}\psi\\
&=&\omega_{\varphi}(\nabla^{\varphi}_{X_{\varphi_t}}X_{\varphi_s},X_{\psi})-
\omega_{\varphi}(\nabla^{\varphi}_{X_{\varphi_s}}X_{\varphi_t},X_{\psi}) \\
   &=& \omega_{\varphi}([X_{\varphi_t},X_{\varphi_s}],X_{\psi}) \\
   &=& \omega_{\varphi}(\{\varphi_t,\varphi_s\},X_{\psi})\\
   &=& -\{\{\varphi_t,\varphi_s\},\psi\}
\end{eqnarray*}
In the line three we use the fact that the Levi-Civita connection is torsion free. In the line four we use the fourth property in lemma \ref{poi}, in the last line we use the second property in lemma \ref{poi}.
We calculate the sectional curvature as follow:
\begin{eqnarray*}
  K_{\varphi}(\varphi_t,\varphi_s) &=& <<R_{\varphi}(\varphi_t,\varphi_s)\varphi_t,\varphi_s>>_{\varphi}\\
   &=& \int_{\Omega}R_{\varphi}(\varphi_t,\varphi_s)\varphi_t\varphi_s (dd^c\varphi)^{n} \\
   &=& -\int_{\Omega}\{\{\varphi_t,\varphi_s\},\varphi_t \}\varphi_s (dd^c\varphi)^n\\
  &=&-\int_\omega \{\varphi_t,\varphi_s\}\{\varphi_t,\varphi_s\}(dd^c\varphi)^n\\
  &=&-||\{\varphi_t,\varphi_s\}||^2_{\varphi}
\end{eqnarray*}
We use in line three the expression of the curvature tensor, in the line four we use fifth property in lemma \ref{poi}.
\end{proof}
\begin{defi}
We say a connection  $D$ in  $\mathcal{H}$  is locally symmetric if its  curvature tensor is parallel i.e   $D R=0$.
\end{defi}

\begin{thm} 
The  Mabuchi space  $\mathcal{H}$   provided by  the Levi-Civita connection  $D$ is a locally symmetric space.
\end{thm}
\begin{proof}
Let $\varphi(t,s,r)$ be 3-parameters family in $\H$. 
\begin{eqnarray*}
   D_r(R_{\varphi}(\varphi_t,\varphi_s)\psi)&=&D_r(-\{\{\varphi_t,\varphi_s\},\psi\}  \\
  &=&-\{D_r\{\varphi_t,\varphi_s\},\psi\} -\{\{\varphi_t,\varphi_s\},D_r\psi\}\\
   &=& -\{\{D_r\varphi_t,\varphi_s\}+\{\varphi_t,D_r\varphi_s\},\psi\}-\{\{\varphi_t,\varphi_s\},D_r\psi\}\\
&=&-\{\{\{D_r\varphi_t,\varphi_s\}\}-\{\varphi_t,D_r\varphi_s\},\psi\}-\{\{\varphi_t,\varphi_s\},D_r\psi\} \\
&=& R_{\varphi}(D_r\varphi_t,\varphi_s)\psi+
R_{\varphi}(\varphi_t,D_r\varphi_s)\psi+R_{\varphi}(\varphi_t,D_r\varphi_s)(D_r\psi)
\end{eqnarray*}
We use the expression of the curvature tensor and  the sixth property in the  lemma \ref{poi} of the Poisson bracket.
Therefore
{\footnotesize
$$
(D_rR_{\varphi})(\varphi_t,\varphi_s)\psi = D_r(R_{\varphi}(\varphi_t,\varphi_s)\psi)-R_{\varphi}(D_r\varphi_t,\varphi_s)\psi-R_{\varphi}(\varphi_t,D_r\varphi_s)\psi-R_{\varphi}(\varphi_t,\varphi_s)(D_r\psi)=0,
$$
}
hence  $\mathcal{H}$ is locally symmetric.
\end{proof}
 \section{The Dirichlet problem}
We now study the regularity of geodesics using pluripotential  theory, the tools using are developed by Bedford and Taylor \cite{BT76,BT82}.
   \subsection{Semmes trick}
   We are interested in the boundary value problem for the geodesic equation: given $\varphi_0$,$\varphi_1$ two  distinct points in $\mathcal{H}$, can one find a path $(\varphi(t))_{0\leq t\leq 1}$ in $\mathcal{H}$ which is a solution of\eqref{1} with end points $\varphi(0)=\varphi_0$ and  $\varphi(1)=\varphi_1$?
For each path $(\varphi_t)_{t\in [0,1]}$ in $\mathcal{H}$, we set
$$\Phi(z,\zeta)=\varphi_t(z)\;, z\in \Omega\;\;and\;\; \zeta=e^{t+is}\in A=\{\zeta\in \mathbb{C} / 1<|\zeta|<e\} $$
We will show in this section that the geodesic equation in $\mathcal{H}$ is equivalent to Monge-Amp\`ere equation on $\Omega\times A$ as in Semmes \cite{Sem92}.
\begin{lem} The  Monge-Amp\`ere  measure of the function $\Phi$ in $\Omega \times A$ is :
\begin{eqnarray*}
(dd^c_{z,\zeta}\Phi(z,\zeta))^{n+1}&=&(dd^c_z\Phi(z,\zeta))^{n+1}+(n+1)(dd^c_z\Phi(z,\zeta))^n\wedge R\\
&&+ {n(n+1)\over 2}(dd^c_z\Phi(z,\zeta))^{n-1}\wedge R^2
\end{eqnarray*}
  
with $$R=R(z,\zeta)= d_zd^c_{\zeta}\Phi+d_{\zeta}d^c_z\Phi+d_{\zeta}d^c_{\zeta}\Phi,$$
\end{lem}
\begin{proof}
We write $d_{z,\zeta}\Phi=d_z\Phi+d_{\zeta}\Phi$ and $d^c_{z,\zeta}\Phi=d^c_z \Phi+d^c_{\zeta}\Phi$, and we give also the expression of   $dd^c_{z,\zeta}\Phi(z,\zeta)$ in $\Omega\times A$. Indeed 
\begin{eqnarray*}
 dd^c_{x,z}\Phi&=&(d_z+d_{\zeta})(d^c_z\Phi+d^c_{\zeta}\Phi)\\
 &=& d_zd_z^c\Phi+d_zd^c_{\zeta}\Phi+d_{\zeta}d^c_z\Phi+d_{\zeta}d^c_{\zeta}\Phi\\
 &=& d_zd_z^c\Phi+R(z,\zeta)
\end{eqnarray*}
with $R= d_zd^c_{\zeta}\Phi+d_{\zeta}d^c_z\Phi+d_{\zeta}d^c_{\zeta}\Phi$ such that $R^3=0$. Then we can find  the expression of$(dd ^c_{x,z}\Phi)^{n+1}$ in $\Omega\times A$. Indeed 
\begin{eqnarray*}
  (dd^c_{z,\zeta}\Phi)^{n+1}&=&(dd^c_z\Phi+R)^{n+1} \\
  &=& \sum_{j=0}^{n+1}C_{n+1}^j(dd^c_z\Phi)^j\wedge(R)^{n+1-j} \\
  &=& (dd^c_z\Phi)^{n+1}+(n+1)(dd^c_z\Phi)^n\wedge R\\
  &&+{n(n+1)\over 2}(dd^c_z\Phi)^{n-1}\wedge R^2
  \end{eqnarray*}
On the second line we use Leibniz formula and the fact that  $R^3=R\wedge R\wedge R=0$ on the third line.
  \end{proof}
\begin{thm}\label{geo}
$(\varphi_t)_{0\leq t\leq 1}$ is a geodesic if and only if $(dd_{z,\zeta}^c\Phi(z,\zeta))^{n+1}=0$.
\end{thm}
\begin{proof}
 From the previous lemma, we have
\begin{eqnarray*}
(dd^c_{z,\zeta}\Phi(z,\zeta))^{n+1}&=&(dd^c_z\Phi(z,\zeta))^{n+1}+(n+1)(dd^c_z\Phi(z,\zeta))^n\wedge R\\
&&+ {n(n+1)\over 2}(dd^c_z\Phi(z,\zeta))^{n-1}\wedge R^2
\end{eqnarray*}
The first term in RHS  of the last equation equal to $0$ a cause of bi-degree. 
We have $$d_{\zeta}\Phi=\partial_{\zeta }\Phi +\bar{\partial}_{\zeta} \Phi=\frac{\partial \Phi}{\partial \zeta}d\zeta +\frac{\partial \Phi}{\partial \bar{\zeta}}d\bar{\zeta}=\dot{\varphi_t}(z)(d\zeta+d\bar{\zeta})$$ and $$d^c_{\zeta}\Phi=\frac{i}{2}(\bar{\partial}\Phi -\partial \Phi)=\frac{i}{2}(\frac{\partial\Phi}{\partial \bar{\zeta}}d\bar{\zeta}-\frac{\partial \Phi}{\partial \zeta}d\zeta)=\frac{i}{2}\dot{\varphi}_t(z)(d\zeta-d\bar{\zeta})$$
and we have also $d_{\zeta}d^c_{\zeta}\Phi=i\ddot{\varphi_t}(z)d\zeta\wedge d\bar{\zeta}$, which gives 
$$R=i\ddot{\varphi_t}(z)d\zeta\wedge d\bar{\zeta}+\frac{i}{2}d_{z}\dot{\varphi_t}\wedge d\bar{\zeta}-\frac{i}{2}d_z\dot{\varphi_t}\wedge d\zeta+d^c_z\dot{\varphi_t}\wedge d\zeta+d^c_z\dot{\varphi_t}\wedge d\bar{\zeta}$$
and 
$$R^2= 2i d_z\dot{\varphi}_t\wedge d^c_z\dot{\varphi}_t\wedge d\zeta\wedge d\bar{\zeta}  $$
Now we can explain the second term also. Indeed
\begin{eqnarray*}
  (dd^c_z\Phi)^n\wedge R &=& (dd_z^ c\varphi_t(z))^n\wedge(i\ddot{\varphi_t}(z)d\zeta\wedge d\bar{\zeta}+\frac{i}{2}d_{z}\dot{\varphi_t}\wedge d\bar{\zeta}\\
  &&-\frac{i}{2}d_z\dot{\varphi_t}\wedge d\zeta+d^c_z\dot{\varphi_t}\wedge d\zeta+d^c_z\dot{\varphi_t}\wedge d\bar{\zeta}) \\
   &=& i\ddot{\varphi_t}(dd^c_z\varphi_t)^n\wedge d\zeta\wedge d\bar{\zeta}
\end{eqnarray*}
And also for third term, we have 
\begin{eqnarray*}
  (dd^c_z\Phi)^{n-1}\wedge R^2 &=& (dd^c\varphi_t(z))^{n-1}\wedge R\wedge R \\
&=&(dd^c\varphi_t(z))^{n-1}\wedge 2i d_z\dot{\varphi}_t\wedge d^c_z\dot{\varphi}_t\wedge d\zeta\wedge d\bar{\zeta}\\
&=&- 2i d_z\dot{\varphi}_t\wedge d^c_z\dot{\varphi}\wedge(dd^c\varphi_t(z))^{n-1}\wedge d\zeta\wedge d\bar{\zeta}
\end{eqnarray*}
After the previous equations we have,
{\footnotesize
\begin{eqnarray*}
  (dd^c_{z,\zeta}\Phi)^{n+1}&=& (n+1)(dd^c_z\Phi(z,\zeta))^n\wedge R+ {n(n+1)\over 2}(dd^c_z\Phi(z,\zeta))^{n-1}\wedge R^2\\
&=& i(n+1)(\ddot{\varphi_t}(dd_z^c\varphi_t)^n-n d_z\dot{\varphi}_t\wedge d^c_z\dot{\varphi_t}\wedge(dd^c\varphi_t(z))^{n-1}\wedge d\zeta\wedge d\bar{\zeta}\\
&=&i(n+1)\left(\ddot{\varphi_t}-\frac{n d_z\dot{\varphi}_t\wedge d^c_z\dot{\varphi_t}\wedge(dd^c\varphi_t(z))^{n-1}}{(dd^c_z\varphi_t)^n}\right)(dd^c_z\varphi_t)^n\wedge d\zeta\wedge d\bar{\zeta}
\end{eqnarray*}
}
From the fact that  $nd(\dot{\varphi_t})\wedge d^c(\dot{\varphi_t})\wedge (dd^c\varphi_t)^{n-1}=\ddot{\varphi_t}(dd^c\varphi_t)^n$, we infer that $\varphi_t$ is geodesic if and only if
$$(dd^c_{z,\zeta}\Phi(z,\zeta)^{n+1}=0.
$$
\end{proof}
After the previous theorem we deduce that the geodesic problem  in Mabuchi space is equivalent to the following Dirichlet problem:
 \begin{equation*}
   \left\{
     \begin{array}{ll}
       (dd^c_{z,\zeta}\Phi(z,\zeta))^{n+1}=0 & \hbox{$\Omega\times A$} \\
       \Phi(z,\zeta)=\varphi_0(z) & \hbox{$ \Omega\times \{|\zeta|=1\}$} \\
       \Phi(z,\zeta)=\varphi_1(z) & \hbox{$\Omega\times \{|\zeta|=e\}$} \;\;(3)\\
  \Phi(z,\zeta)=0 & \hbox{$\partial \Omega\times A$}
     \end{array}
   \right.
\end{equation*}
   \subsection{Continuous envelopes}
 We have that $\varphi_0$ and $\varphi_1$ are smooth, in the sequel we can assume that $\varphi_0$ and $\varphi_1$ are only $C^{1,1}$.
\begin{defi}
The  Perron-Bremermann envelope is defined by  $$\Phi(z,\zeta)=\sup\{u(z,\zeta) \in\mathcal{F}(\Psi,\Omega\times A)\}$$ with
  $$\mathcal{F}(\Psi,\Omega \times A)=\{u\in PSH(\Omega\times A)\cap C^{0}(\bar{\Omega}\times \bar{A})\;/\;\;u^{*}\leq \Psi\;  \;on \;\partial(\Omega\times A)\}$$
Where $\Psi|_{\partial\Omega \times \bar{A }}=0$ \;\;and\;\; $\Psi _{\partial A\times \Omega}=\left\{
                                                                                 \begin{array}{ll}
                                                                                   \varphi_{0}(z)& \hbox{$\{|\zeta|=1\}$;} \\
                                                                                  \varphi_1(z), & \hbox{$\{|\zeta|=e\}$.}
                                                                                 \end{array}
                                                                               \right.$\\
.
\end{defi}
\begin{thm} \label{zero}If $\Psi\in C^{0}(\partial(\Omega\times A))$. Then the Perron-Bremermann envelope  $\Phi$ satisfies the following conditions:\\
    i)\;$\Phi\in PSH(\Omega\times A)\cap C^{0}(\bar{\Omega}\times \bar{A})$.\\
    ii) \;$\Phi|_{\partial(\Omega\times A)}=\Psi$.\\
    iii)\;$( dd_{z,\zeta}^c\Phi(z,\zeta))^{n+1}=0 $\;\;in\;\; $\Omega\times A$.
\end{thm}
\begin{proof}

Let $\rho$ be a strictly  plurisubharmonic defining  of $\Omega=\{ \rho<0\}$. Observe that the family  $\mathcal{F}(\Psi,\Omega \times A)$ is not empty .\\
i)We start   by  proving  the plursubharmonicity  of $\Phi$ in $\Omega\times A$.
We can write the Dirichlet  problem on  following way:
\begin{equation*}\left\{
    \begin{array}{ll}
      (dd^c_{z,\zeta}\Phi(z,\zeta))^{n+1}=0 & \hbox{$\Omega \times A$}\;\;  \\
      \Phi(z,\zeta)=\Psi(z,\zeta)& \hbox{$\partial (\Omega\times A)$}
    \end{array}
  \right.
\end{equation*}
with  $\Psi(z, \zeta)={1\over e^2-1}(\varphi_1(z)(|\zeta|^2-1)-\varphi_0(z)(|\zeta|^2-e^2))$.
 Let $h\in Har(\Omega \times A)\cap C^0(\bar{\Omega}\times \bar{A})$ be a harmonic function in $\Omega \times A$, continuous up to the boundary  of $\Omega\times A$, the solution of the following Dirichlet  problem
      $$\left\{
          \begin{array}{ll}
            \Delta_{z,\zeta} h(z,\zeta)=0, & \hbox{$\Omega \times A$} \\
            h=\Psi, & \hbox{$\partial (\Omega \times A)$}
          \end{array}
        \right.
      $$
    Exists, since $\Omega \times A$ is a regular domain.\\
For all $v\in\mathcal{F}(\Psi,\Omega \times A)$, we have $v^{*}\leq  \Psi$ on  $\partial (\Omega \times A)$,which implies $$(v-h)^{*}\leq  0\; on  \;\partial (\Omega \times A)$$ Furthermore we have $$\Delta_{z,\zeta} (v-h)(z,\zeta)=\Delta_{z,\zeta} v(z,\zeta)\geq 0\;\;in  \;\;\Omega\times A$$ Then by maximum principle
   $$v(z,\zeta)\leq h(z,\zeta) \;\; in \;\; \Omega \times A$$ the last inequality  holds for every  function in  $\mathcal{F}(\Psi,\Omega \times A)$ , hence it holds for upper envelope of subsolution
$$\Phi(z,\zeta)\leq h(z,\zeta)\;\;in \;\; \Omega \times A$$
It also holds  for its  upper semi-continuous regularization on  the boundary  $(\Omega \times A)$, we get  $$(\Phi(z,\zeta))^*\leq \Psi(z,\zeta)\;\; on\;\; \partial (\Omega\times A)$$ and consequently $$\Phi^*\in \mathcal{F}(\Psi, \Omega\times A)$$
   Since the  function  $\Phi^*$ is  plurisubharmonic  in  $\Omega \times A$ and $$\Phi(z,\zeta)\leq \Phi(z,\zeta))^* \;\; in \;\; \Omega\times A$$ we infer that  $$(\Phi(z,\zeta))^*= \Phi(z,\zeta)\;\; in \;\; \Omega \times A.$$ hence  $\Phi$ is plurisubharmonic function in  $\Omega \times A$.\\
Since   $\Phi$ is plurisubharmonic in $\Omega \times A$, implies  $\Phi$ is  upper semi continuous. We now prove it is lower upper semi-continuous. Indeed
Fix $\epsilon>0$ and since   $\partial(\Omega\times A)=(\partial \Omega\times \bar{A})\cup(\bar{\Omega}\times \partial A)$ is compact  and the function $\Psi$ is  continuous on   $\partial(\Omega\times A)$, we can choose  $\beta > 0$ so small that $$(z,\zeta)\in\Omega \times A,  \forall(\xi, \eta)\in\partial(\Omega\times A)
\Vert(z,\zeta)-(\xi,\eta)\Vert\leq \beta\Rightarrow \vert\Phi(z,\zeta)-\Psi(\xi,\eta)\vert\leq \epsilon$$
Fix  $a=(a_1,a_2)\in \mathbb{C}^{n}\times\mathbb{C}$ with  $\Vert a\Vert\leq \beta $. So, We have  the following inequality
\begin{equation*}
    \Phi(\xi+a_1,\eta+a_2)\leq \Psi(\xi,\eta)+\varepsilon \;\; if\; (\xi,\eta)\in(\Omega\times A \setminus\{a\})\cup\partial(\Omega\times A)
\end{equation*}
and
\begin{equation*}
    \Phi^{*}(z+a_1,\zeta+a_2)\leq \Psi(z+\alpha,\zeta+a_2)+\varepsilon\leq \Phi(z,\zeta)+\varepsilon\;\; if\; \Omega \times A\cap\partial((\Omega\times A) \setminus\{a\})
\end{equation*}
It follows that the function
{\footnotesize
$$W(z,\zeta)=\left\{
  \begin{array}{ll}
    \max(\Phi(z,\zeta),\Phi(z+a_1,\zeta+a_2)-2\varepsilon\, & \hbox{$(z,\zeta)\in(\Omega\times A)\setminus(\Omega\times A) \setminus\{a\} $;} \\
   \Phi(z,\zeta), & \hbox{$(z,\zeta)\in(\Omega\times A)\cap(\Omega\times A) \setminus\{a\}$.}
  \end{array}
\right.$$
}
is plurisubharmonic in  $\Omega\times A$
because \\
1) If  $(z,\zeta)\in(\Omega\times A)\cap(\Omega\times A) \setminus\{a\}$ it coincides with $\Phi$ which is plurisubharmonic.\\
2) If  $(z,\zeta)\in(\Omega\times A)\setminus(\Omega\times A) \setminus\{a\}$,  it is the  maximum of two   plurisubharmonic functions .\\
3) After the two previous  inequalities, we infer that the  function  $W$ coincides on the boundary, furthermore $$W\leq \Psi\;\; on \;\;\partial (\Omega\times A)$$ Which implies  $W\in \mathcal{F}(\Omega\times A,\Psi)$,  finally we get $$\Phi(z+a_1,\zeta+a_2)-2\varepsilon \leq \Phi(z,\zeta)\;for  \;(z,\zeta)\in\Omega\times A\; and \;a\in \mathbb{C}^{n+1},||a||\leq \beta$$
Thus   $\Phi$ is lower  semi-continuous, therefore it is continuous .\\
ii)We are going to prove that  $$\lim\limits_{\Omega\times A\ni(z,\zeta)\rightarrow (\xi_0,\eta_0)\in   \partial(\Omega\times A)} \Phi(z,\zeta)=\Psi(\xi_0,\eta_0)$$ Firstly, since  $\Phi\in \mathcal{F}(\Psi,\Omega\times A)$ we have   $$\limsup_{(z,\zeta)\to(\xi_0,\eta_0)}\Phi(z,\zeta)\leq \Psi(\xi_0,\eta_0)\; \forall(\xi_0,\eta_0)\in \partial (\Omega \times A)$$
To prove the reverse of  inequality, we  construct a plurisubharmonic  barrier  function at each point $(\xi_0,\eta_0)=\gamma_0\in \partial(\Omega\times A)$. Since  $\rho$ is strictly plursubharmonic function,  we can choose $B$ large enough so that the function  $$ b(z,\xi):=B\rho(z)-|z-\xi_0|^2-|\zeta-\eta_0|^2$$ is  plurisubharmonic  in $\Omega\times A$ and continuous up to the boundary such that $b(\xi_0,\eta_0)\leq0$ with $b <0$ for all $ (z,\zeta)\in\bar{\Omega} \times \bar{A}\setminus \gamma_0$. \\
Fix $\epsilon>0$ and take $\eta>0 $ such that 
      $\Psi(\gamma_0)-\epsilon \leq \Psi(\gamma )\; \forall \gamma \in\partial (\Omega \times A)$ and $|\gamma-\gamma_0|\leq \eta$.
       We choose a big constant $C$ so that  $$Cb+\Psi(\gamma_0)-\varepsilon\leq \Psi\;\; on\;\; \partial(\Omega \times A)$$This implies that  the  function $V(z,\zeta)=Cb(z,\zeta)+\Psi(\gamma_0)-\varepsilon\in PSH(\Omega\times A)$ is    $$V\leq\Psi\;\;on \;\;\partial(\Omega\times A)$$ Thus we have  $V\in \mathcal{F}(\Psi,\Omega\times A)$  which  implies $V(z,\zeta)\leq \Phi(z,\zeta)$ in $\Omega \times A$. We get 
       $$\Psi(\xi_0,\eta_0)-\varepsilon\leq \liminf_{(z,\zeta)\to(\xi_{0};\eta_{0})}\Phi(z,\zeta)$$
   therefore
       $$\lim_{(z,\zeta)\to(\xi_{0};\eta_{0})}\Phi(z,\zeta)=\Psi(\xi_0,\eta_0) \; \forall (\xi_0,\eta_0)\in\partial(\Omega\times A)$$
       iii)
        The Perron Bremermann envelope $$\Phi(z,\zeta)=\sup \{u(z,\zeta)\in \mathcal{F}(\Omega\times A,\Psi)\}$$ is plurisubharmonic continuous up the boundary of  $\Omega\times A$ and  \;$ \Phi|_{\partial(\Omega\times A)}=\Psi$.\\
By a lemma due to Choquet, this  envelope can be realised by a countable family  $$\Phi(z,\zeta)=\sup\{u(z,\zeta)\in\mathcal{F}(\Omega\times A,\Psi)\}=\sup_{j} \{u_{j}(z,\zeta)\in\mathcal{F}(\Omega\times A,\Psi)\}$$
We put  $$\Phi_j(z,\zeta)=\max(u_1(z,\zeta),u_2(z,\zeta)......u_j(z,\zeta))\nearrow \Phi(z,\zeta)$$ the function  $\Phi_j$is increasing and$$(\Phi(z,\zeta))^{*}=(\sup_{j}\{\Phi_j(z,\zeta)\})^*$$
Let  $\mathbb{B}\subset\subset \Omega \times A$ be any  ball, we consider the following Dirichlet problem  $$\left\{
                        \begin{array}{ll}
                          (dd^c(u_j(z,\zeta))^{n+1}=0, & \hbox{$\mathbb{B}$;} \\
                          u_j=\Phi_j, & \hbox{$\partial \mathbb{B}$.}
                        \end{array}
                      \right.$$
since  $$(dd^c _{z,\zeta}u_j(z,\zeta))^{n+1}\leq (dd^c_{z,\zeta}\Phi_j(z,\zeta))^{n+1}\quad\hbox{in}\quad\mathbb{B}$$ and
$$
u_j=\Phi_j\quad\hbox{on}\quad\partial\mathbb{B}
$$
we have   $$\Phi_j(z,\zeta)\leq u_j(z,\zeta)\;\;in\;\; \mathbb{B}.$$
We consider  the following function  $$\Theta(z,\zeta)=\left\{
                                   \begin{array}{ll}
                                     u_j(z,\zeta), & \hbox{$(z,\zeta)\in \mathbb{B}$;} \\
                                     \Phi_j(z,\zeta), & \hbox{$(z,\zeta)\in \bar{\Omega}\times \bar{A}\setminus\mathbb{B}$.}
                                   \end{array}
                                 \right.$$

The function $\Theta_j$ belongs  to $\mathcal{F}(\Omega\times A,\Psi)\}$. This implies  $$\Theta_j (z,\zeta)\leq \Phi_j(z,\zeta)\;in \; \Omega\times A$$ furthermore  $$\Theta_j=\Phi_j=\Psi\; on \; \partial(\Omega\times A)$$ then  $$u_j(z,\zeta)=\Phi_j(z,\zeta)\; in \; \mathbb{B}$$  therefore
$$(dd^c_{z,\zeta}(\Phi_j(z,\zeta)))^{n+1}=(dd^c_{z,\zeta}(u_j(z,\zeta)))^{n+1}=0\; in \; \mathbb{B}$$ since  $\mathbb{B}$ is arbitrary we give  $$(dd^c_{z,\zeta}\Phi_j(z,\zeta))^{n+1}=0\; in  \;\Omega\times A$$ By the   continuity property of  Monge-Am\`epre operators of  Bedford and Taylor along monotone sequences, we have
$$(dd^c_{z,\zeta}(\Phi_j(z,\zeta))^{n+1}\longrightarrow (dd^c_{z,\zeta}(\Phi(z,\zeta))^{n+1}=0$$
i.e  $$(dd^c_{z,\zeta}(\Phi(z,\zeta))^{n+1}=0 \; \;in\; \;  \Omega\times A.$$
\end{proof}
  \subsection{Lipschitz regularity}
  In this subsection we will give the geodesic regularity Lipschitz in time and in space. We begin by regularity Lipschitz in time. We use a barrier argument as noted by Berndtsson \cite{Bern13}
 \begin{prop}The Perron Bremermann  envelope $\Phi(z,\zeta)=\sup\{u(z,\zeta)/ u \in \mathcal{F}(\Omega\times A,\Psi\}$  is Lipschitz function with respect to $t=\log|\zeta|$.
 \end{prop}
 \begin{proof} The proof follows from a classical balayage technique. Indeed, we consider the following  function   $$\chi(z,\zeta)=\max(\varphi_0(z)-A\log|\zeta|,\varphi_1(z)+A(\log|\zeta|-1))$$
where  $A>0$ is  a big constant. Furthermore
\begin{eqnarray*}
  \chi(z,\zeta)|_{\Omega\times \{|\zeta|=1\}} &=& \max(\varphi_0(z),\varphi_1(z)-A)=\varphi_0 (z)\\
  \chi(z,\zeta)|_{ \Omega\times \{|\zeta|=e\}} &=& \max(\varphi_0(z)-A,\varphi_1(z))=\varphi_1(z) \\
    \chi(z,\zeta)|_{\partial \Omega\times A } &=& \max(-A\log|\zeta|, A(\log|\zeta|-1)\leq 0
\end{eqnarray*}
The last line follows by  $\varphi_0=\varphi_1=0 $ on  $\partial \Omega$ and  $1<|\zeta|<e$. Then   $\chi$ it belongs to $\mathcal{F}(\Omega\times A,\Psi)$ and
  $$\chi(z,\zeta)\leq \Phi(z,\zeta) \; in\; \Omega\times A$$
since  $\Phi(z,\zeta)=\varphi(z,\log|\zeta|)$ and $\chi(z,\zeta)=\chi(z,\log|\zeta|)$, which implies
  $${\varphi(z,\log|\zeta|)-\varphi(z,1)\over \log|\zeta|}\geq {\chi(z,\zeta)-\varphi(z,1)\over \log|\zeta|}={\chi(z,\zeta)-\chi(z,1)\over \log|\zeta|}$$
$$\lim _{|\zeta|\rightarrow 1}{\chi(z,\zeta)-\chi(z,1)\over \log|\zeta|}=\lim _{|\zeta|\rightarrow 1}{\varphi_0(z)-A\log(|\zeta|)-\varphi_0(z)\over \log|\zeta|}=-A$$
which gives  $\dot{\varphi}(z,0)\geq -A$, similarly for other case  $\dot{\varphi}(z,1)\leq A$. Since the function $\varphi_t$ is  convex along  $t$ (by subharmonicity in $\zeta$),we infer that for almost everywhere $z$,$t$,
$$-A\leq \dot{\varphi}(z,0)\leq \dot{\varphi}(z,t)\leq \dot{\varphi}(z,1)\leq A$$
then  $\varphi_t$ is  uniformly  Lipschitz at $t=\log|\zeta|$.
 \end{proof}
 We will prove the regularity Lipschitz in space by adapting the method of Bedford and Taylor \cite{BT76}(see also  \cite{GZ17}).
  \begin{thm}\label{lip}
The Perron Bremermann  envelope $\Phi(z,\zeta)=\sup\{u(z,\zeta)/ v \in \mathcal{F}(\Omega\times A,\Psi\}$ is Lipschitz function up to the boundary with respectively to space variable.
\end{thm}
\begin{proof} Let $\rho$  be a smooth defining of $\Omega$ which is strictly psh in a neighbourhood $\Omega'$ of $\Omega$, and also $\alpha$  be a smooth defining of $A$ which is strictly psh in a neighbourhood $A'$ of $A$.
   We will construct an extension of  function  defined on the boundary of  $\Omega\times A$ by
       $$
      \Psi(z,\zeta)= \left\{
         \begin{array}{ll}
           \varphi_0(z) & \hbox{$\Omega\times \{|\zeta|=1\}$} \\
            \varphi_1(z) & \hbox{$\Omega\times \{|\zeta|=e\}$} \\
           0 & \hbox{$\partial\Omega\times A$}
         \end{array}
       \right.
       $$
  Let  $\chi$ be a   smooth function with  compact  support defined  in  $[0,1]$ by $\chi(t)=1$ near of  $0$ and by $\chi(t)=0$ near of  $1$. We put  $$\tilde{\chi}(\zeta)=\chi(\log|\zeta|))$$
  is a smooth function in $\bar{A}$. We have $\tilde {\chi}(\zeta)=1$\;near of $|\zeta|=1$ and  $\tilde {\chi}(\zeta)=0$\; near of  $|\zeta|=e$.\\
  We consider the following function:
  $$ F(z,\zeta)=\tilde{\chi}(\zeta)\tilde{\varphi}_0(z,\zeta)+(1-\tilde{\chi}(\zeta))\tilde{\varphi}_1(z,\zeta)+B\alpha(\zeta),$$
 where $\tilde{\varphi}_0(z,\zeta)=\varphi_0(z)$, $\tilde{\varphi}_1(z,\zeta)=\varphi_1(z)$, 
The  function $F$ satisfies   $$F|_{\Omega\times \partial A}=\left\{
                                                             \begin{array}{ll}
                                                               \varphi_0(z), & \hbox{$\Omega\times \{|\zeta|=1\}$} \\
                                                               \varphi_1(z), & \hbox{$\Omega\times \{|\zeta|=e\}$}\\
                                                           0,& \hbox{$\partial {\Omega}\times A$}
                                                             \end{array}
                                                           \right.$$
  The function $F$ is extension  plurisubharmonic  of the function $\Psi$ defined in  $\Omega \times \partial A$ to $\Omega \times A$. We can also extend the function $\Psi$ defined in $\partial \Omega\times A$ by putting 
 $$F(z,\zeta)=D\rho(z),$$
 where $D$ is a big constant.
 
 On two cases the function $F$ satisfies the following properties  $$F\leq\Phi \;on \;\partial(\Omega\times A)\quad\hbox{and}\quad(dd^c_{z,\zeta}F)^{n+1}\geq (dd^c_{z,\zeta}\Phi)^{n+1} \;in \;\Omega\times A$$
  By maximum Principle we get
$$F(z, \zeta)\leq \Phi(z,\zeta)\quad\hbox{in}\quad\Omega\times A$$
Applying the same process to the boundary data $-\Psi$ we choose  $C^{1,1}$ function defined in $\Omega \times A$ such that $G=-\Psi $ on $\partial(\Omega\times A)$, the maximum Principle implies 
$$\Phi(z,\zeta)\leq -G(z,\zeta) \quad\hbox{in}\quad \Omega\times A$$
After the two  previous inequalities we have
$$
F(z,\zeta)\leq \Phi(z,\zeta)\leq -G(z,\zeta)\quad\hbox{in}\quad \Omega\times A
$$
Since $F(.,\zeta)\leq \Phi(,\zeta)$ in $\Omega $, the envelope $\Phi(,\zeta)$ can be extended respectively to variable $z$ as a plurisubharmonic  function in $\Omega'$ by setting $\Phi(,\zeta)=F(,\zeta)$ in $\Omega'\setminus \Omega$ with $\zeta$ fixed in $A$. Fix $\delta>0$ so small that $z+h\in \Omega$ whenever $z\in \bar{\Omega}$ and $||h||<\delta$, this set noted in sequel  by $\Omega_h$. Fix $h\in  \mathbb{C}^n$ such that  $||h||<\delta$. Recall that $F$ and $G$ are Lipschitz in each variable, thus 
$$|F(z+h,\zeta)-F(z,\zeta)|\leq C||h||\;and \;|G(z+h,\zeta)-G(z,\zeta)|\leq C||h||$$
for any $z\in \bar{\Omega}$ and $\zeta \in \bar{A}$.\\
Observe that the function  $v(z,\zeta):=\Phi(z+h,\zeta)-C||h||$ is well defined psh in the open set $\Omega\times A$. If $z\in\partial\Omega\cap\Omega_h$ and $\zeta\in A$, then
$$v(z,\zeta)=\Phi(z+h,\zeta)-C||h||\leq -G(z+h,\zeta)-C||h|| \leq -G(z,\zeta)=\Psi(z,\zeta).$$
 If $z\in\Omega\cap\partial \Omega_h$ and $\zeta\in A$, then
 $$v(z,\zeta)=\Phi(z+h,\zeta)-C||h||\leq F(z+h,\zeta)-C||h|| \leq F(z,\zeta)\leq\Phi(z,\zeta).$$
This shows that the function $w$ defined by
$$w(z,\zeta):=
\left\{
\begin{array}{ll}
  \max(v(z,\zeta),\Phi(z,\zeta)) & \hbox{if $(z,\zeta)\in \Omega\cap \Omega_h \times A $} \\
  \Phi(z,\zeta)  & \hbox{if $(z,\zeta)\in \Omega\setminus \Omega_h \times A $}
   \end{array}
    \right.
   $$          
 is plurisubharmonic in $\Omega\times A$. Since $w\leq \Psi$ on $\partial(\Omega\times A)$ we get $w\leq\Phi$                       in $\Omega\times A$, hence $v\leq \Phi$ in $\Omega\times A$. We have shown that 
 $$\Phi(z+h,\zeta)-\Phi(z,\zeta)\leq C ||h||$$
 whenever $z\in \Omega\cap \Omega_h$ , $||h|| \leq \delta$ and  $\zeta \in A$. Replacing $h$ by $-h$ shows that 
 $$|\Phi(z+h,\zeta)-\Phi(z,\zeta)|\leq C ||h||$$
 Which proves that $\Phi(, \zeta)$ is Lipschitz in every $z\in\bar{\Omega}$.
\end{proof}
 \section{Case of the unit ball}
 In this section we shall show how to use the proof of Bedford and Taylor \cite{BT76}, which is simplified by Demailly \cite{Dem93} in the unit ball for giving the regularity in space variable for our geodesics problem. We need some preparation for prove this regularity. The open  subset $\B:=\{z\in \C^n\; /\; |z_1|^2+|z_2|^2.....+|z_n|^2 < 1\}$ of $\C^n$ is called the unit ball. First we shall define the Mobius transformation of the unit ball. Let $a\in \B\setminus \{0\}\subset \C^n$. Denote the orthogonal projection onto the subspace of in  $\C^n$ generated
by the vector $a$ by, $$P_a(z):=\frac{<z,a>a }{||a||^2}.$$
The Mobius transformation associated with a is the
mapping 
$$T_a(z):=\frac{P_a(z)-a +\sqrt{(1-||a||^2)}(z-P_a(z))}{1-<z,a>}$$
With $<z,a>=\sum_{i=1}^n z_i \bar{a}_i$ denote the hermitian scalar product of $z$ and $a$. 
 For every $a\in B $, the Mobius transformation has the following properties\\
i)\; $T_a(0)=-a$ and $T_0(a)=0$.\\
ii)an elementary computation yields 
\begin{equation}\label{auto}T_a(z)=z-a+<z,a>a+O(||a||^2)=z-h+O(||a||^2),
\end{equation}
with  $h=h(a,z):=a-<z,a>z$ and  $O(||a||^2)$ is uniformly with respect of  $z\in \mathbb{\bar{B}}$.\\
We need in the sequel the following useful lemma for giving the regularity in unit ball.
\begin{lem}Let $u$ be a  plurisubharmonic function in domain  $\Omega\subset\subset \C^n$, assume that there exists $B, \delta>0$ such that
$$u(z+h)+u(z-h)-2u(z)\leq B||h||^2,\;\; \forall 0<||h||<\delta$$
and for all $z\in \Omega$ and $dist(z,\Omega)>\delta$. Then $u$ is $C^{1,1}$-smooth , ant its second derivative, which exists almost everywhere, satisfies $$||D^2u||_{L^{\infty}(\Omega)}\leq B.$$\end{lem}
\begin{proof}
Let  $u_{\varepsilon}=u\ast\chi_{\varepsilon}$  denote the standard  regularization of  $u$ defined  in $\Omega_{\varepsilon}=\{z \in \Omega \; / dist(\partial \Omega,z)> \varepsilon \}$ for $0<\varepsilon<<1$. Fix $\delta>0$ small enough  and  $0<\varepsilon<\frac{\delta}{2}$. Then  for $ 0<||h||<\frac{\delta}{2}$ we have 
\begin{eqnarray}
u_{\varepsilon}(z+h)+u_{\varepsilon}(z-h)-2u_{\varepsilon}(z)\leq B||h||^2
\end{eqnarray}
It follows from Taylors formula that if $z\in \Omega_{\epsilon}$
$$\frac{d^2}{dt^2}u_{\varepsilon}(z+th)|_{t=0}:=\lim_{t\to 0^+}\frac{u_{\varepsilon}(z-th)+u_{\varepsilon}(z+th)-2u_{\varepsilon}(z)}{t^2}$$
therefore by  having
$D^2u_{\varepsilon}(z).h^2\leq B||h||^2$ for all  $z\in \Omega_{\varepsilon}$ and $h\in \mathbb{C}^n$. Now for  $z\in \Omega_{\varepsilon}$
$$D^2u_{\varepsilon}(z).h^2=\sum_{i,j=1}^n\left( \frac{\partial^2 u_{\varepsilon}}{\partial z_i\partial z_j}h_ih_{j}+2\frac{\partial^2 u_{\varepsilon}}{\partial z_i\partial\bar{z}_j}h_i\bar{h}_{j}+\frac{\partial^2 u_{\varepsilon}}{\partial \bar{z}_i\partial \bar{z}_j}\bar{h}_{i}\bar{h}_{j}\right)$$
Recall that $u_{\epsilon}$ is plurisubharmonic in  $\Omega_{\varepsilon}$ hence
$$D^2u_{\varepsilon}(z).h^2+D^2u_{\varepsilon}(z).(ih)^2=4\sum_{i,j=1}^n\frac{\partial^2 u_{\varepsilon}}{\partial z_i\partial\bar{z}_j}h_i\bar{h}_{j}\geq 0.$$
The above upper-bound also yields a lower-bound of  $D^2u_{\varepsilon}$
$$D^2u_{\varepsilon}.h^2\geq -D^2u_{\varepsilon}.(ih)^{2}\geq -B||h||^2 $$
For any  $z\in \Omega$ and  $h\in \mathbb{C}^n$. This implies that
$$||D^2u_{\varepsilon}||_{L^{\infty}(\Omega)}\leq B$$
Thus, we have shown that $D u_{\varepsilon}$ is uniformly Lipschitz in  $\Omega_{\varepsilon}$. We infer that $Du$ is Lipschitz in  $\Omega$ and $Du_{\varepsilon}\longrightarrow Du$  uniformly in  compact subsets of $\Omega$. Since the dual of  $L^1$ is $L^{\infty}$, it follows from the Alaoglu-Banach theorem that, up to extracting a subsequence,there exists a bounded function $V$ such that  $D^2u_{\varepsilon} \longrightarrow V$  weakly in  $L^{\infty}$. Now $D^2u_{\varepsilon}\longrightarrow D^2 u$ in the sense of distributions hence $V=D^2u$. Therefore  $u$ is  $C^{1,1}$ in $\Omega$ and  its second-order derivative
exists almost everywhere with $||D^2u||_{L^{\infty}(\Omega)}\leq B$.
\end{proof}

\begin{thm}Let $\mathbb{B}$ be the  unit ball in $\mathbb{C}^n$. Let $\varphi_0$ and  $\varphi_1$ be the end geodesic  points which are $C^{1,1}$. Then the Perron-Bremermann envelope
$$\Phi(z,\zeta)=\sup\{u(z,\zeta)/ u \in \mathcal{F}(\Omega\times A,\Psi)\}$$
admits second-order partial derivates almost everywhere with respect to variable $z\in\B$ which  locally bounded uniformly with respect to $\zeta \in A$  , i.e  for any compact subset  $K\subset \mathbb{B}$ there exists  $C$ which depend on $K,\varphi_{0}$ and  $\varphi_{1}$ such that   $$\Vert D^2_z\Phi\Vert_{ L^{\infty}(K\times A)}\leq C.$$
\end{thm}
\begin{proof}
 For proving the theorem, we weed to prove the following inequality
$$\Phi(z+h,\zeta)+\Phi(z-h,\zeta)-2\Phi(z,\zeta)\leq A||h||^2,$$
for any  $||h ||<<1$ , $z\in\mathbb{B}$ and $\zeta \in A$.\\
The idea is to study the boundary behavior of the plurisubharmonic  function $(z,\zeta)\longmapsto \frac{1}{2}(\Phi(z+h,\zeta)+\Phi(z-h,\zeta)$ in order to compare it with the function $\Phi$ in $\B\times A$.
This does not make sense since the translations do not preserve the
boundary. We are instead going to move point $z$ by automorphisms of the unit
ball: the group of holomorphic automorphisms of the latter acts transitively
on it and this is the main reason why we prove this result for the unit ball
rather than for a general strictly pseudoconvex domain (which has generically
few such automorphisms).\\
By the fact $\Phi$ is Lipschitz with respectively to   $z$ variable (theorem \ref{lip}) and expansion (\ref{auto}) we have $$|\Phi(T_a(z),\zeta)-\Phi(z-h,\zeta)|\leq C||T_a(z)-(z-h)||\leq C||a||^2$$
and $$|\Phi(T_{-a}(z),\zeta)-\Phi(z+h,\zeta)|\leq C||T_{-a}(z)-(z+h)||\leq C||a||^2$$
which  implies 
$$\Phi(z+h,\zeta)+\Phi(z-h,\zeta)\leq \Phi(T_a(z),\zeta)+\Phi(T_{-a}(z),\zeta)+2C||a||^2.$$
We consider the following functions:
$$F(z,\zeta):=\Phi(T_a(z),\zeta)+\Phi(T_{-a}(z),\zeta)+2C||a||^2,$$
and $G(z,\zeta)=2\Phi(z,\zeta)+D||a||^2$,
we Observe that the functions  $F$ and $G$ are  well defined in  $\mathbb{B}\times A$ and plurisubharmonic  in $\mathbb{B}\times A$. We need to show that 
$$F(z,\zeta)\leq G(z,\zeta)\quad\hbox{in}\quad\mathbb{B}\times A.$$
For showing the last inequality  we will apply the maximum  Principle, then we need to prove $$F(z,\zeta)\leq G(z,\zeta)\quad\hbox{on}\quad \partial(\mathbb{B}\times A)$$ and
$$(dd^c_{z,\zeta}F(z,\zeta))^{n+1}\geq (dd^c_{z,\zeta}G(z,\zeta))^{n+1} \quad\hbox{in}\quad\mathbb{B}\times A$$
The last  inequality is easy follows from the fact that $F$ is a plurisubharmonic  and $(dd^c_{z,\zeta}\Phi)^{n+1}=0$ in $\mathbb{B}\times A$ by (theorem \ref{zero}).\\
 We need to compare $F$ and $G$ in the boundary  of $\mathbb{B}\times A$. Indeed, since  $\partial(\mathbb{B}\times A)=(\partial\mathbb{B}\times \bar{A})\cup(\bar{\mathbb{B}}\times \partial{A})$, then we will compare in two parts, we begin by the part  $\partial\mathbb{B}\times \bar{A}$, in this part we get
$$F|_{\partial{B}\times \bar{A}}=2C||a||^2 \;\;and  \;\;G|_{\partial\mathbb{B}\times \bar{A}}=D||a|^2.$$
For shows that  $F|_{\partial\mathbb{B}\times \bar{A}}\leq G_{\partial\mathbb{B}\times \bar{A}}$, we take just $$D=2C.$$
For the second part  $\bar{\mathbb{B}}\times \partial{A}$, We compare  in $\B\times \bar{A}$ only ,because $\partial \B\times \bar{A}$ belongs to the previous part, since $\partial{A}=\{|\zeta|=1\}\cup \{|\zeta|=e\}$, we begin this part by comparing in case $\B\times  \{|\zeta|=1\}$, we have 
$$F|_{\B\times \{|\zeta|=1\}}=\varphi_0(T_a(z))+\varphi_0(T_{-a}(z))+2C||a||^2$$
and $$G|_{\B\times \{|\zeta|=1\}}=2\varphi_0(z)+D||a||^2$$
We apply Taylor expansion and we get
$$\varphi_0(T_a(z))=\varphi_0(z-h+O(|a|^2)=\varphi_0(z)-d\varphi(z).h+O(|a|^2)$$
and
$$\varphi_0(T_{-a}(z))=\varphi_0(z+h+O(|a|^2)=\varphi_0(z)+d\varphi(z).h+O(|a|^2)$$
Which is implies 
  $$\varphi_0(T_a(z))+ \varphi_0(T_{-a}(z))\leq 2\varphi_0(z)+2C_0|a|^2$$
  where $C_0$ depend only on the $\varphi_0$
  then $$F(z,\zeta) \leq 2\varphi_0(z)+2C_1|a|^2+2C|a|^2$$
  If we take $D=2(C_0+C)$, we get 
 $ F(z,\zeta) \leq G(z,\zeta)$ on $\B\times \{|\zeta|=1\}$. By same methods we get $F(z,\zeta) \leq G(z,\zeta)$ on $\B\times \{|\zeta|=1\}$ for $D=2(C_1+C)$, where $C_1$ depend only on the $\varphi_1$ which concludes  the second part.\\
 By part one and two we infer 
 $$F(z,\zeta)\leq G(z,\zeta)\quad \hbox{in}\quad \partial(\mathbb{B}\times A)$$
From the maximum Principle we get
$$F(z,\zeta)\leq G(z,\zeta)\quad\hbox{in}\quad \mathbb{B}\times A$$
Which is implies 
\begin{eqnarray*}
\Phi(z+h,\zeta)+\Phi(z-h,\zeta)-2\Phi(z,\zeta)&\leq&\Phi(T_a(z),\zeta)+\Phi(T_{-a}(z),\zeta)\\
&&+2C||a||^2-2\Phi(z,\zeta)\\
&\leq& \Phi(T_a(z),\zeta)+\Phi(T_{-a}(z),\zeta)\\&&+2C||a||^2-2\Phi(z,\zeta)\\
&\leq& D||a||^2
\end{eqnarray*}
Observe that the mapping $a\longmapsto h(a,z)=a-<z,a>z$ is a local diffeomorphism in 
neighborhood of the origin as long as $||z||<1$, which  depend on $z\in \mathbb{B}$ smoothly and its inverse $h\longmapsto a(h,z)$ is linear  with a norm less than or equal to $\frac{1}{1-||z||^2}$ since
$$||h||\geq ||a||-||a||||z||^2=||a||(1-||z||^2)$$
which gives
$$\Phi(z+h, \zeta)+\Phi(z-h,\zeta)-2\Phi(z,\zeta)\leq \frac{D ||h||^2}{(1-||z||^2)^2}$$
Fix a compact set $K\subset \mathbb{B}$ compact, there exists  $\delta>0$ such that  $\forall z\in K$ and  $\forall 0<||h||<\delta $ we have
$$\Phi(z+h, \zeta)+\Phi(z-h,\zeta)-2\Phi(z,\zeta)\leq \frac{D||h||^2}{dist(K,\partial\mathbb{B})^2}$$
after the previous lemma we get
$$||D^2_z\Phi||_{L^{\infty}(K\times A)}\leq D$$
with  $C= \frac{D}{dist(K,\partial\mathbb{B})^2}$.
\end{proof}

\section{Moser-Trudinger inequalities}
 In this section we assume $\Omega$ is pseudoconvex circled domain.
   We try to solve the complex Monge-Amp\`ere equation
  \begin{equation} 
  (dd^c \f_t)^n=\frac{e^{-t \f_t} \mu}{\int_{\Omega} e^{-t \f_t} d\mu }
  \end{equation}
  with $\f_t$ smooth and plurisubharmonic, ${\f_t}_{|\partial \Omega}=0$ and $\mu$ just the Lebesgue  normalised so that  $\mu(\Omega)=1$.
  We know that
\begin{itemize}
\item We can solve this equation if $t$ is not too large ($t=1$ is treated in \cite{GKY13}
and even $t< (2n)^{1+1/n}(1 + 1/n)^{(1+1/n)}$).
\item One can not solve the equation if $t$ is too large, cf \cite[section 6.2]{GKY13} and \cite{BB11}.
\end{itemize}
 We denote by
$$E(\varphi):={1\over n+1}\int_{\Omega}\varphi (dd^c\varphi)^n$$
the Monge-Amp\`ere  energy functional of a plurisubharmonic function $\varphi$,  which is defined as the  primitive of Monge-Amp\`ere operator. The expression 
$$\mathcal{F}_t(\varphi):=E(\varphi)+{1\over t}\log\left[\int_{\Omega}e^{-t\varphi}d\mu\right]$$
defines the Ding functional.
\begin{defi}
 We say  the functional $\mathcal{F}_t$ is coercive , if there exist $\varepsilon>0$ and $B>0$ such that :
$$\mathcal{F}_t(\varphi)\leq \varepsilon E(\varphi)+B \; \;\;\;\forall \varphi\in \mathcal{H}$$
\end{defi}
\begin{defi}
Set $\Phi_s(z)=\Phi(z,e^s)$. The continuous family $(\Phi_s)_{0\leq s\leq 1}$ is called the
geodesic joining $\varphi_0$ and $\varphi_1$.
\end{defi}
We show that $E$ is linear along of geodesics, this result is in \cite[lemma 22]{GKY13}, and was proven by Rashkovskii \cite{Rash16} in the Cegrell class, we reprove it for continuous geodesics for convenience of the reader. 
\begin{lem}\label{21}Let $(\Phi_s)_{0\leq s\leq 1}$ be a continuous geodesic. Then $s\longmapsto E(\Phi_s)$ is affine.
\end{lem}
\begin{proof} After the proof of theorem \ref{geo} we have 
{\footnotesize
\begin{eqnarray*}
 (dd^c_{z,\zeta}\Phi(z,\zeta))^{n+1}&=&(n+1)(dd^c_z\Phi(z,\zeta))^n\wedge R+ {n(n+1)\over 2}(dd^c_z\Phi(z,\zeta))^{n-1}\wedge R^2\\
 &=& (n+1)\left( d_{\zeta}d^c_{\zeta}\Phi\wedge(d_zd_z^c\Phi)^n-n d_z d^c_{\zeta}\Phi\wedge d_{\zeta} d_z^c\Phi\wedge (d_z d_z^c\Phi)^{n-1}\right)
\end{eqnarray*}
 }
We have by definition of $E$
$$E(\Phi(.,\zeta)=\frac{1}{n+1}\int_{\Omega}\Phi(z,\zeta)(d_z d_z^c\Phi(z,\zeta))^n$$
Which implies
$$d^c_{\zeta}E=\frac{1}{n+1}\int_{\Omega}d^c_{\zeta}\Phi\wedge (d_z d_z^c\Phi)^{n}$$
{\footnotesize
\begin{eqnarray*}
d_{\zeta}d^c_{\zeta}E(\Phi)&=&\frac{1}{n+1}\left(\int_{\Omega} d_{\zeta}d_{\zeta}^c\Phi\wedge (d_z d_z^c\Phi)^{n-1}+n\int_{\Omega}d^c_{\zeta}\Phi\wedge d_{\zeta}d_z d_z^c\Phi\wedge (d_z d_z^c\Phi)^{n}\right)\\
&=&\frac{1}{n+1}\left(\int_{\Omega}d_{\zeta}d^c_{\zeta}\Phi\wedge (d_z d_z^c\Phi)^{n-1}-n\int_{\Omega} d_z d^c_{\zeta}\Phi\wedge d_{\zeta} d_z^c\Phi\wedge (d_z d_z^c\Phi)^{n-1}\right)\\
&=&\frac{1}{(n+1)^2} \int_{\Omega}(dd^c_{z,\zeta}\Phi)^{n+1}
\end{eqnarray*}
}
where the second equality follows from Stokes theorem because $ d_{\zeta}\Phi=0$ on $\partial \Omega$, and the
last one be above calculation.\\
Thus, it follows from theorem \ref{zero} that $\zeta \in A\longmapsto E(\Phi(.,\zeta)\in \mathbb{R}$ is harmonic in $\zeta$. Since $\Phi$ is invariant by rotation with respect to $\zeta$, hence it is affine in $t=\log|\zeta|$.
\end{proof}
We recall here  \cite[proposition 23] {GKY13}.
\begin{prop}\label{33} Assume that $\Omega$ is circled, let  $\varphi_t$ be an  $S^1$-invariant solution of $(MA)_t$. Then
$$\mathcal{F}_t(\varphi_t)=\sup_{\psi \in I(\Omega)}\mathcal{F}_t({\psi}),$$
where $I(\Omega)$ denotes  all $S^1$-invariant plurisubharmonic functions $\psi$ in $\Omega$ which are continuous up to the boundary,
with zero boundary value.
\end{prop}
\begin{proof}Let $(\Phi)_{0\leq s\leq 1}$ be a  geodesic joining $\Phi_0:=\varphi_t$ to $\Phi_1=\psi$. It follows from work of Berndtsson \cite{Bern06} that
$$s \longmapsto -{1\over t}\log\left(\int _{\Omega}e^{-t\Phi_s}d\mu\right)$$
is convex, since $ s\longmapsto E(\Phi_s)$ is affine from lemma \ref{21}. Then
$ s\longmapsto \mathcal{F}(\Phi_s)$ is concave.\\
therefore  it is sufficient to show that the derivative of $\mathcal{F}_t(\Phi_s)$ at $s=0$ is non-negative to conclude $\mathcal{F}_t(\varphi_t)=\mathcal{F}(\Phi_0)\geq\mathcal{F}_t(\Phi_s)$ for all s, in particular at $s=1$ where it yields $ \mathcal{F}_t(\varphi_t)\geq \mathcal{F}_t(\psi)$. When $\longmapsto\Phi_s$ is smooth, a direct computation yields, for $s = 0$,
$${d\over ds}\mathcal{F}_t(\Phi_s)=\int_{\Omega}\dot{\Phi}_s\left[(dd^c\Phi_s)^n-{e^{-t\Phi_s}\mu\over\int_{\Omega}e^{-t\Phi_s}d\mu}\right]=0$$
For the general case, the same method as in the proof of \cite[theorem 6.6]{BBGZ} applies. 
\end{proof}
\begin{lem}\label{3}
The Functional $\mathcal{F}_{t}$ is upper semi-continuous in $\mathcal{E}^1_{C}(\Omega)=\{\psi\in \mathcal{E}^1(\Omega)/\psi=0 \;on \;\partial \Omega\; and  \;E(\psi)\geq -C\}$.
\end{lem}
\begin{proof}
Recall $\mathcal{F}_{t}(\psi)= E(\psi)+\frac{1}{t}\log(\int _{\Omega}e^{-t\psi}d\mu)$. The first term is upper semi-continuous in $\mathcal{E}^1(\Omega)$. For the second term we apply Skoda uniform integrability theorem\cite{Zer01}.\\ Assume without loss of generality that $t=1$. We need to check that 
$\psi\in\mathcal{E}_{C}^1(\Omega)\longmapsto \int_{\Omega} e^{-\psi}d\mu $ is upper semi-continuous.\\
Let $\psi_j$ be a sequence in  $\mathcal{E}_{C}^1(\Omega)$ converging to $\psi$ these functions have zero Lelong number. The following extension:\\\ $g_j=\psi_j+\psi$ to $\Omega\subset K\subset \Omega'$  as $\tilde{g}_{j}= g_j$ in $\Omega$, $\tilde{g}_{j}=0$ in $\Omega'\setminus \Omega$. We apply  Skoda's uniform  integrability estimates:
$$\int_{\Omega}e^{-2(\psi+\psi_j)}d\mu\leq \int_{K}e^{-2(\psi+\psi_j)} d\mu \leq C $$

\begin{eqnarray*}
|\int_{\Omega}e^{-\psi_j}d\mu-\int_{\Omega}e^{-\psi}d\mu|\leq \int_{\Omega}|\psi-\psi_j|e^{-(\psi_j+\psi)}d\mu\leq C||\psi_j-\psi||_{L^2(\mu)}.
\end{eqnarray*}
as follows from the Cauchy-Schwarz inequality and the elementary inequality
$$|e^{a}-e^{b}|\leq |a-b|e^{a+b},\; for\; all\; a,b\geq 0 $$
The conclusion follows since $(\psi_j)$ converges to $\psi$ in $L^2(\mu)$.
\end{proof}
We recall that the  Dirichlet problem $(MA)_t$ has a solution for $t=1$ by \cite{GKY13}, we moreover have  uniqueness if $\Omega$ is stricltly $\varphi$-convex( $\Omega$ is strictly convex dor the metric $dd^c\varphi$). We recall here the main result of  \cite{GKY13}.
\begin{thm}Let $\Omega\subset \mathbb{C}^n$ be a bounded smooth strongly pseudoconvex domain which is
circled. Let $\varphi$ be a smooth $S^1$-invariant strictly plurisuharmonic  solution of the complex MongeAmp\`ere problem $(MA)_1$. If $\Omega$ is strictly $\varphi$-convex, then $\varphi$ is the unique $S^1$-invariant solution of $(MA)_1$.
\end{thm}
Inspired by Dinezza-Guedj \cite[theorem 5.5]{DG16}, we now prove the following theorem 
\begin{thm}Let $\Omega\subset \mathbb{C}^n$ be a smooth strongly pseudo-convex  circled domain. If there exists $\varepsilon(t), M(t)>0$ such that,
$$\mathcal{F}_{t}(\psi)\leq \varepsilon(t) E(\psi)+M(t)\;\; \;\;\forall \psi \in \mathcal{H},$$
then $(MA)_t$ admits a $S^1$-invariant smooth strictly plurisubharmonic function solution.\\
Conversely if $(MA)_t$ admits such a solution $\varphi_t$ and $\Omega$ is strictly $\varphi_t$-convex, then there exists $\varepsilon(t), M(t)>0$ such that,
$$\mathcal{F}_{t}(\psi)\leq \varepsilon(t) E(\psi)+M(t)\;\; \;\;\forall \psi \in \mathcal{H}.$$
\end{thm}
\begin{proof}If we assume  the following inequality holds,
$$\mathcal{F}_t(\psi)\leq \varepsilon(t) E(\psi)+M(t)$$
 then the same method of  \cite{GKY13}applies, if only we change $\varphi$ by $t\varphi$.\\
Conversely, as $\varphi_t$ is a solution of $(MA)_t$ then from the (proposition \ref{33}) we have 
\begin{equation}\label{11}
\mathcal{F}_t(\varphi_t):=\sup\{\mathcal{F}_t(\psi)/ \psi\in \mathcal{H}\cap I(\Omega)\}
\end{equation}
assume for contradiction that there is  no $\varepsilon>0$ such that
$$\mathcal{F}_t(\psi)\leq \varepsilon E(\psi)+M$$
for all $\psi \in \mathcal{H}$. Put $\varepsilon_j=\frac{1}{j}$ and  $M=\mathcal{F}_t(\varphi_t)+1$. Then we can find a sequence $(\varphi_j)\subset \mathcal{H}$ such that
$$\mathcal{F}_t(\varphi_j)>\frac{E(\varphi_j)}{j}+\mathcal{F}_t(\varphi_t)+1$$
We discuss here two cases, the first case if  $E(\varphi_j)$ does not blow up to $-\infty$, we reach a contradiction, by letting $j$ go to $+\infty$. Indeed we can assume that $E(\varphi_j)$ bounded and $\varphi_j$ converges  to  some $\psi\in \mathcal{E}^1(\Omega)$ which is $S^1$-invariant. Since $\mathcal{F}_t$ is upper semi-continuous by lemma \ref{3}, we infer $\mathcal{F}_t(\psi)\geq \mathcal{F}_t(\varphi_t)+1>\mathcal{F}_t(\varphi_t)$ contradiction because $\varphi_t$ is the solution of $(MA)_t$.\\
The second case if $E(\varphi_j)\rightarrow-\infty$. It follows that $d_j=-E(\varphi_j)\rightarrow+\infty$.\\ We let $(\phi_{s,j})_{0\leq s\leq d_j} $ denote the weak geodesic joining $\varphi_t$ to $\varphi_j$ and set $\psi_j:=\phi_{1,j}$. We know that is  $s\longmapsto E(\phi_{s,j})$ is affine along of the Mabuchi geodesic. Thus  
$E(\phi_{s,j})=a_js+b_j$, where  $a_j$ and $b_j$ are real numbers. For $s=0$  we have 
$$E(\phi_{0,j})=b_j=E(\varphi_t)$$
and for $s=d_j$ we have
$$E(\varphi_j)=E(\phi_{d_j,j})=a_jd_j+E(\varphi_t)$$
therefore  $a_j=\frac{E(\varphi_j)-E(\varphi_t)}{d_j}$. Then 
\begin{equation}\label{e}
E(\phi_{s,j})= \frac{E(\varphi_j)-E(\varphi_t)}{d_j}s+E(\varphi_t)
\end{equation}
 Since  $s\longmapsto E(\phi_{s,j})$ is affine along of the Mabuchi geodesic and by Berndtsson \cite{Bern06} convexity result, we infer that  the map $s \longmapsto \mathcal{F}_t(\phi_{s,j})$ is concave, which implies with (\ref{11}) that
  $$ 0\geq\mathcal{F}_t(\phi_{1,j})-\mathcal{F}_t(\phi_{0,j})\geq \frac{\mathcal{F}_t(\phi_{d_j,j})-\mathcal{F}_t(\phi_{0,j})}{d_j}>-\frac{1}{j}+\frac{1}{d_j}$$
thus $\mathcal{F}_t(\psi_j)\longrightarrow \mathcal{F}_t(\varphi_t)$. This shows that $(\psi_j)$ is a maximizing sequence for $\mathcal{F}_t$. If we take $t=1$ on equation (\ref{e}) we get
\begin{equation}\label{f}
E(\psi_j)=\frac{E(\varphi_j)-E(\varphi_t)}{d_j}+E(\varphi_t)=-1-\frac{E(\varphi_t)}{d_j}+E(\varphi_t)\geq -1+E(\varphi_t)
\end{equation}
Passing to subsequence, we can assume that $\psi_j$ converge to $\psi\in \mathcal{E}^1(\Omega)$ which is $S^1$-invariant. Since $\mathcal{F}_t$ is upper semi-continuous and $\psi_j$ is a maximizing sequence for $\mathcal{F}_t$
then we have $\mathcal{F}_t(\psi)=\mathcal{F}_t(\varphi_t)$ and so $\psi=\varphi_t$ thanks to the uniqueness. Letting  $j$ to infinity in (\ref{f}) we get 
$$E(\psi)=-1+E(\varphi_t)$$
This yields a contradiction.
\end{proof}
  
\end{document}